\documentclass[11pt,a4paper]{article}
\usepackage[hmargin={20mm,20mm},vmargin={20mm,20mm}]{geometry}
\usepackage{graphicx}
\graphicspath{ {./images/} }
\usepackage{url}

\usepackage{amsmath}
\usepackage{amsthm}
\usepackage{amssymb}
\usepackage{authblk}
\usepackage[utf8]{inputenc}
\usepackage{csquotes}
\usepackage[english]{babel}

\usepackage{caption}
\usepackage{subcaption}
\usepackage{graphicx}

\allowdisplaybreaks

\newtheorem{thm}{Theorem}[section]

\newtheorem{corollary}[thm]{Corollary}
\newtheorem{proposition}[thm]{Proposition}
\newtheorem{example}[thm]{Example}

\DeclareMathOperator{\updim}{\overline{dim}}
\DeclareMathOperator{\lodim}{\underline{dim}}

\title{On the intermediate dimensions of concentric spheres and related sets}
\author{Justin T. Tan}
\date{}

\begin{document}

\maketitle

\begin{abstract}
    The intermediate dimensions are a family of dimensions introduced in 2019 by Falconer, Fraser, and Kempton to interpolate between the Hausdorff dimension and the box dimension. To date, there are limited examples of explicit calculations of the intermediate dimensions of interesting sets. We calculate the intermediate dimensions of sets of concentric spheres converging to the origin in Euclidean spaces. We also consider related sets including isolated points on concentric spheres and attenuated topologist's sine curves.
    
    \vspace{5mm}
    
    \noindent\emph{Mathematics Subject Classification} 2020: 28A80.

    \noindent\emph{Key words and phrases}: intermediate dimensions, box dimension, Hausdorff dimension, concentric spheres, topologist's sine curve.
\end{abstract}

\section{Introduction}

The recent introduction of spectra of dimensions has provided a new approach in fractal geometry, revealing interesting results and useful applications (see \cite{interpolating} for a survey on dimension interpolation). The first widely-studied example of dimension interpolation is the Assouad spectrum, introduced in 2016 by Fraser and Yu \cite{spectra} to interpolate between the upper box dimension and the Assouad dimension. The intermediate dimensions are another example of interpolation between well-known dimensions. Introduced in 2019 by Falconer, Fraser, and Kempton \cite{intermediate}, the intermediate dimensions take values between the Hausdorff dimension and box dimensions (also known as the Minkowski dimension). We refer the reader to \cite{fractal} for a detailed treatment of fractal geometry, including the Hausdorff and box dimensions.

For $\theta\in(0,1]$ and bounded $E\subseteq\mathbb{R}^d$, the upper and lower $\theta$-intermediate dimensions of $E$ are defined as
\begin{multline}\label{upintdim}
    \updim_\theta (E) = \inf\{s\geq0:\forall\varepsilon>0,\exists\delta_0>0,\forall0<\delta<\delta_0,\\
    \exists\text{ a cover }\mathcal{U}\text{ of }E\text{ such that }\delta\leq|U|\leq\delta^{\theta}(\forall U\in\mathcal{U})\text{ and }\sum_{U\in\mathcal{U}}|U|^s\leq\varepsilon\}
\end{multline}
and
\begin{multline}\label{lointdim}
    \lodim_\theta (E) = \inf\{s\geq0:\forall\varepsilon>0,\forall\delta_0>0,\exists0<\delta<\delta_0,\\
    \exists\text{ a cover }\mathcal{U}\text{ of }E\text{ such that }\delta\leq|U|\leq\delta^{\theta}(\forall U\in\mathcal{U})\text{ and }\sum_{U\in\mathcal{U}}|U|^s\leq\varepsilon\}
\end{multline}
respectively. We also define $\updim_0 = \lodim_0 = \dim_H$, the Hausdorff dimension. It is clear that $\updim_1 = \updim_B$ and $\lodim_1 = \lodim_B$, the upper and lower box dimensions respectively. (\ref{upintdim}) and (\ref{lointdim}) can be expressed with the interval of allowed diameters of covering sets being $[\delta^{1/\theta},\delta]$ instead, as was done in \cite{intermediate}, \cite{projection}, and others. The definition of the intermediate dimensions is natural as it allows the sets in the cover to have diameters in some range, thus interposing between the definition of the Hausdorff dimension, which only imposes an upper bound on the diameters of the covering sets, and the definition of the box dimension, which fixes the diameters of all covering sets.

In \cite{intermediate}, it was shown that $\updim_\theta (E)$ and $\lodim_\theta (E)$ are monotonically increasing with respect to $\theta\in[0,1]$ and continuous over $\theta\in(0,1]$. The Hausdorff dimension is countably stable but the $\theta$-intermediate dimensions are closure-invariant for $\theta\in(0,1]$, so continuity at $\theta=0$ is not guaranteed in general.

A key question raised in \cite{intermediate} is, when are the intermediate dimensions of a set continuous at zero? Results relating to this question include \cite[Proposition 4.1]{intermediate}, \cite[Corollaries 6.1, 6.2, 6.4]{projection}, and \cite[Corollaries 3.5, 3.6, 3.7]{brownian}. Intuitively, continuity at zero would imply that the complexity that causes the box dimension to exceed the Hausdorff dimension can be fully accounted for by allowing the covering sets to take diameters in $[\delta,\delta^{\theta}]$ and letting $\theta\to0^+$. Even in cases where the exact formula for the intermediate dimensions of a set may be difficult to calculate, we may be able to determine whether the intermediate dimensions of the set are continuous at zero (such as for the Bedford-McMullen carpets \cite[Proposition 4.1]{intermediate}). We take this approach in Proposition \ref{points}.

Fraser, Falconer, and Kempton \cite[Proposition 3.1]{intermediate} proved that the $\theta$-intermediate dimensions of the set $F_p:=\{n^{-p}:n\in\mathbb{N}\}$ for $p>0$ are $\updim_\theta(F_p)=\lodim_\theta(F_p)=\frac{\theta}{p+\theta}$ for all $\theta\in[0,1]$. In this case, we observe non-trivial interpolation between the Hausdorff and box dimensions, as well as continuity at zero for the intermediate dimensions. This set will be the starting point for the more complicated examples that we will work with. Recent work calculating the intermediate dimensions of other sets include sharper bounds for the Bedford-McMullen carpets \cite{bmcarpets} and exact values for elliptical polynomial spirals \cite{elliptical}.

Within this note, $|U|$ will denote the diameter of the set $U\subseteq\mathbb{R}^d$, $|x|$ will denote the Euclidean norm of the point $x\in\mathbb{R}^d$, and $B(x,r)$ will denote the closed ball of radius $r>0$ centered at $x\in\mathbb{R}^d$.

\section{Intermediate dimensions of sets of concentric spheres in $\mathbb{R}^d$}

For $p>0$, let
\begin{equation*}
    C_p^d:=\left\{x\in\mathbb{R}^d:|x|\in F_p\right\}.
\end{equation*}
That is, $C_p^d$ is the union of $(d-1)$-spheres in $\mathbb{R}^d$ with radii from $F_p=\left\{n^{-p}:n\in\mathbb{N}\right\}$. Therefore, we can write
\begin{equation*}
    C_p^d=\bigcup_{i\in\mathbb{N}}\mathcal{S}_{1/i^p}^{d-1},
\end{equation*}
where $\mathcal{S}_{r}^{d-1}=\left\{x\in\mathbb{R}^{d}:|x|=r\right\}$. We illustrate $C_1^2$ in Figure \ref{concentricgraph}. Since the case in $\mathbb{R}$ will just be two copies of $F_p$, we will only focus on $d\geq2$.

\begin{figure}[h]
\centering
\includegraphics[width=5cm]{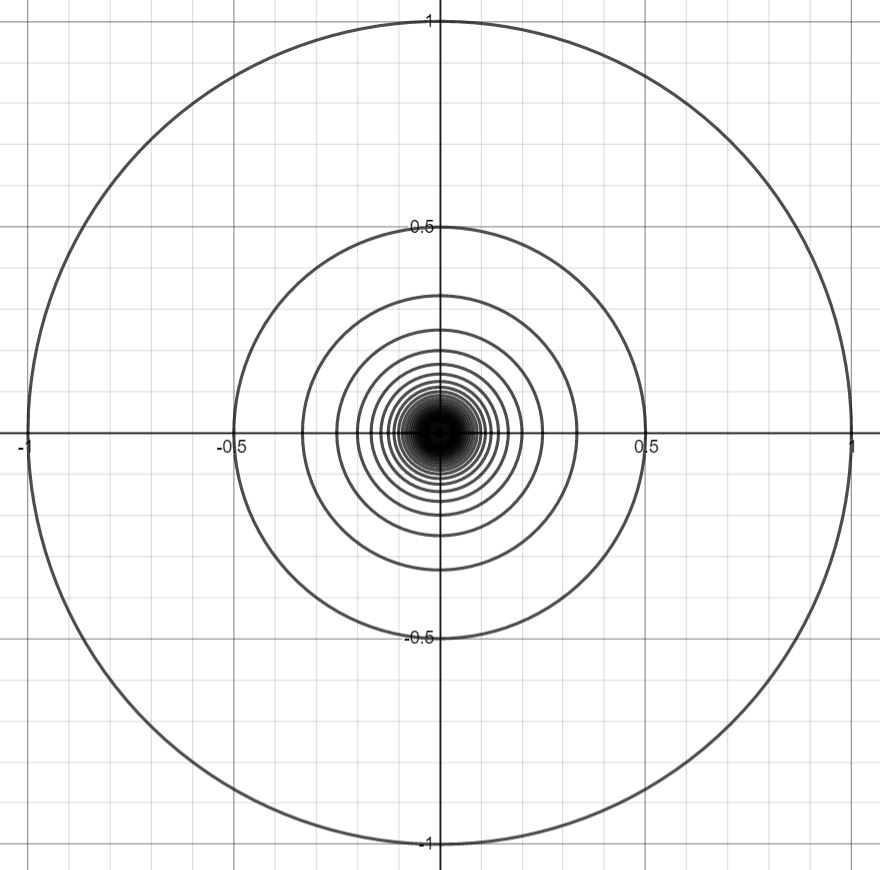}
\caption{The set of concentric circles $C_1^2\subseteq\mathbb{R}^2$.}
\label{concentricgraph}
\end{figure}

\begin{thm}\label{concentric1}
Let $d\geq2$. For $p>0$ and $0\leq\theta\leq1$,
\begin{equation}
    \lodim_\theta \left(C_p^d\right) = \updim_\theta \left(C_p^d\right) =
    \begin{cases}
        \frac{dp(d-1)+d\theta(1-p(d-1))}{dp+\theta(1-p(d-1))} &\text{if } 0<p<\frac{1}{d-1}\\
        d-1 &\text{if } p\geq\frac{1}{d-1}
    \end{cases}.
\end{equation}
\end{thm}

\begin{proof}
Since the Hausdorff dimension is countably stable, $\lodim_0 \left(C_p^d\right) = \updim_0 \left(C_p^d\right) =d-1$. We now focus on $\theta\in(0,1]$.

\emph{Case 1: $p\geq\frac{1}{d-1}$.}
Due to monotonicity with respect to $\theta$, $\lodim_\theta\left(C_{p}^{d}\right)\geq d-1$.

For the upper bound, let $0<\delta<1$ and let $M=\left\lceil\delta^{-\frac{1}{1+p}}\right\rceil$. We can cover $\bigcup_{i=M}^{\infty} \mathcal{S}_{1/i^p}^{d-1} \subseteq B\left(0,\frac{1}{M^p}\right)$ with a grid of at most $\left(\frac{2\sqrt{d}}{M^p \delta}+1\right)^{d}$ $d$-cubes of side length $\delta/\sqrt{d}$ (and thus diameter $\delta$). By the binomial theorem, $\left(\frac{2\sqrt{d}}{M^p \delta}+1\right)^{d} = \sum_{k=0}^{d}\binom{d}{k}\left(\frac{2\sqrt{d}}{M^p \delta}\right)^{k}$.

There is a constant $\xi_d$ dependent only on $d$ such that for any $R>r>0$, we can cover $\mathcal{S}_{R}^{d-1}$ by at most $\xi_d\left(\frac{R}{r}\right)^{d-1}$ sets of diameter $r$. Now, for each $1\leq i\leq M-1$, $\frac{1}{i^p}\geq \frac{1}{(M-1)^p} \geq  \delta^{\frac{p}{1+p}}>\delta$, so we can cover $\mathcal{S}_{1/i^p}^{d-1}$ with at most $\xi_d \left(\frac{1}{i^p \delta}\right)^{d-1}$ sets of diameter $\delta$. Therefore, we can cover $\bigcup_{i=1}^{M-1} \mathcal{S}_{1/i^p}^{d-1}$ with at most $\sum_{i=1}^{M-1} \xi_d \left(\frac{1}{i^p \delta}\right)^{d-1}$ sets of diameter $\delta$. We can simplify this as:
\begin{align*}
    \sum_{i=1}^{M-1} \xi_d \left(\frac{1}{i^p \delta}\right)^{d-1} 
    = \frac{\xi_d}{\delta^{d-1}} + \frac{\xi_d}{\delta^{d-1}} \sum_{i=2}^{M-1} \left(\frac{1}{i^{p(d-1)} }\right)
    &\leq \frac{\xi_d}{\delta^{d-1}} + \frac{\xi_d}{\delta^{d-1}} \int_{x=1}^{M-1} \frac{1}{x^{p(d-1)} } dx\\
    &\leq \frac{\xi_d}{\delta^{d-1}} + \frac{\xi_d}{\delta^{d-1}} \int_{x=1}^{M-1} \frac{1}{x} dx\\
    &= \frac{\xi_d}{\delta^{d-1}} + \frac{\xi_d}{\delta^{d-1}} \log{(M-1)}.\\
\end{align*}

This gives us a $\delta$-cover of $C_p^d$, so for all $0<\delta<1$,
\begin{align*}
    N_{\delta}\left(C_p^d\right) &\leq \sum_{k=0}^{d}\binom{d}{k}\left(\frac{2\sqrt{d}}{M^p \delta}\right)^{k} + \frac{\xi_d}{\delta^{d-1}} + \frac{\xi_d}{\delta^{d-1}} \log{(M-1)}\\
    &\leq \sum_{k=0}^{d}\binom{d}{k}\left(\frac{2\sqrt{d}}{\delta^{-\frac{p}{1+p}} \delta}\right)^{k} + \frac{\xi_d}{\delta^{d-1}} + \frac{\xi_d}{\delta^{d-1}} \log{\left(\delta^{-\frac{1}{1+p}}\right)}\\
    &\leq \sum_{k=0}^{d}\binom{d}{k}\frac{\left(2\sqrt{d}\right)^{k}}{\delta^{\frac{k}{1+\frac{1}{d-1}}}} + \frac{\xi_d}{\delta^{d-1}} + \frac{\xi_d}{(1+p)\delta^{d-1}} \log{\left(\frac{1}{\delta}\right)}\\
    &= \sum_{k=0}^{d}\binom{d}{k}\frac{\left(2\sqrt{d}\right)^{k}}{\delta^{\frac{k}{d}(d-1)}} + \frac{\xi_d}{\delta^{d-1}} + \frac{\xi_d}{(1+p)\delta^{d-1}} \log{\left(\frac{1}{\delta}\right)},\\
\end{align*}
where $N_{\delta}\left(E\right)$ denotes the minimum number of $\delta$-sets required to cover $E\subseteq\mathbb{R}^d$.

We can now bound the upper box dimension of $C_p^d$ from above:
\begin{align*}
    \updim_B\left(C_{p}^{d}\right) &= \limsup_{\delta\to0^+}\frac{\log N_\delta(C_{p}^{d})}{-\log(\delta)}\\
    &\leq \limsup_{\delta\to0^+}\frac{\log \left( \sum_{k=0}^{d}\binom{d}{k}\frac{\left(2\sqrt{d}\right)^{k}}{\delta^{\frac{k}{d}(d-1)}} + \frac{\xi_d}{\delta^{d-1}} + \frac{\xi_d}{(1+p)\delta^{d-1}} \log{\left(\frac{1}{\delta}\right)}\right)}{\log\left(\frac{1}{\delta}\right)}\\
    &= \limsup_{x\to\infty}\frac{\log \left( \sum_{k=0}^{d}\binom{d}{k}\left(2\sqrt{d}\right)^{k}x^{\frac{k}{d}(d-1)} + \xi_d x^{d-1} + \frac{\xi_d}{(1+p)}x^{d-1} \log{\left(x\right)}\right)}{\log\left(x\right)}\\
    &= \limsup_{x\to\infty}\frac{\log \left( \sum_{k=0}^{d}\binom{d}{k}\left(2\sqrt{d}\right)^{k}x^{\left(\frac{k}{d}-1\right) (d-1)} + \xi_d  + \frac{\xi_d}{(1+p)} \log{\left(x\right)}\right)+\log \left(x^{d-1}\right)}{\log\left(x\right)}\\
    &= d-1,
\end{align*}
since $\frac{k}{d}-1\leq0$ within the summation in the first term.

With this, we conclude that for all $p\geq\frac{1}{d-1}$ and all $\theta\in[0,1]$,
\begin{equation}
    d-1\leq \dim_H\left(\mathcal{S}_{1}^{(d-1)}\right) \leq \lodim_\theta\left(C_{p}^{d}\right) \leq \updim_\theta\left(C_{p}^{d}\right) \leq \updim_B\left(C_{p}^{d}\right) \leq d-1.
\end{equation}

\emph{Case 2: $0<p<\frac{1}{d-1}$.}
We start by noting that for $d\in\mathbb{N}$ and $r>0$, $\mathcal{H}^{d-1} \left(\mathcal{S}_{r}^{d-1}\right)=\eta_{d-1}r^{d-1}$, where $\eta_{d-1}$ is a constant that depends only on $d-1$ (see, for example, Chapter 3 of \cite{fractal} for the constants).

Let $s=\frac{dp(d-1)+d\theta(1-p(d-1))}{dp+\theta(1-p(d-1))}$ and $\delta_0=2^{-\frac{1+p}{[1-(1-\theta)(d-s)][1-p(d-1)]}}$. For any given $0<\delta<\delta_0$, write $M=\left\lceil\delta^{-\frac{1-(1-\theta)(d-s)}{1+p}}\right\rceil$ and define
\begin{equation}
    \mu_\delta := \delta^{s-(d-1)} \sum_{i=1}^{M} \mathcal{H}^{d-1}\restriction_{\mathcal{S}_{1/i^p}^{d-1}}.
\end{equation}
Clearly, this is a Borel measure supported on $C_p^d$.

The total mass distributed over $C_p^d$ is:
\begin{align*}
    \mu_\delta\left(C_p^d\right) &= \delta^{s-(d-1)} \sum_{i=1}^{M} \mathcal{H}^{d-1}\restriction_{\mathcal{S}_{1/i^p}^{d-1}}\left(C_p^d\right)\\
    &= \delta^{s-(d-1)} \sum_{i=1}^{M} \mathcal{H}^{d-1}\left(\mathcal{S}_{1/i^p}^{d-1}\right)\\
    &= \delta^{s-(d-1)} \sum_{i=1}^{M} \eta_{d-1}\frac{1}{i^{p(d-1)}}\\
    &\geq \delta^{s-(d-1)} \eta_{d-1} \int_{x=1}^M \frac{1}{x^{p(d-1)}}dx\\
    &= \frac{\delta^{s-(d-1)} \eta_{d-1}}{1-p(d-1)} \left(M^{1-p(d-1)}-1\right)\\
    &\geq \frac{\delta^{s-(d-1)} \eta_{d-1}}{1-p(d-1)} \left(\frac{M^{1-p(d-1)}}{2}\right)\\
    &(\text{since }M^{1-p(d-1)}\geq \delta^{-\frac{1-(1-\theta)(d-s)}{1+p}(1-p(d-1))} > \delta_0^{-\frac{1-(1-\theta)(d-s)}{1+p}(1-p(d-1))} =2)\\
    &\geq \frac{\eta_{d-1}}{2(1-p(d-1))} \left(\delta^{-\frac{1-(1-\theta)(d-s)}{1+p}(1-p(d-1))}\delta^{s-(d-1)}\right)\\
    &= \frac{\eta_{d-1}}{2(1-p(d-1))} \delta^{\frac{s[dp+\theta(1-p(d-1))]-[dp(d-1)+d\theta(1-p(d-1))]}{1+p}}\\
    &= \frac{\eta_{d-1}}{2(1-p(d-1))}
\end{align*}
by our choice of $s$. This is a fixed amount independent of $0<\delta<\delta_0$.

Suppose $U\subseteq\mathbb{R}^d$ is such that $\delta\leq|U|\leq\delta^\theta$. For $1\leq i<j \leq M$, the distance between $\mathcal{S}_{1/i^p}^{d-1}$ and $\mathcal{S}_{\frac{1}{j^p}}^{d-1}$ is at least $\frac{p}{M^{1+p}}$ by the mean value theorem, so $U$ intersects at most $\frac{1}{p}M^{1+p}|U|+1$ of the spheres carrying mass. We observe that for each sphere $\mathcal{S}_{1/i^p}^{d-1}$ carrying mass that $U$ intersects, $\mu_\delta \left(U\cap\mathcal{S}_{1/i^p}^{d-1}\right)=\delta^{s-(d-1)} \mathcal{H}^{d-1}\left(U\cap\mathcal{S}_{1/i^p}^{d-1}\right)\leq \delta^{s-(d-1)}\eta_{d-1}|U|^{d-1}$.

Therefore, the mass carried by $U$ is:
\begin{align*}
    \mu_\delta\left(U\right) &\leq \left(\frac{1}{p}M^{1+p}|U|+1\right)\delta^{s-(d-1)}\eta_{d-1}|U|^{d-1}\\
    &\leq \frac{\eta_{d-1}}{p}\left(\delta^{-\frac{1-(1-\theta)(d-s)}{1+p}}+1\right)^{1+p}\delta^{s-(d-1)}|U|^{d-s}|U|^{s} + \eta_{d-1}\delta^{s-(d-1)}|U|^{d-s-1}|U|^{s}\\
    &\leq \frac{\eta_{d-1}}{p}\left(2\delta^{-\frac{1-(1-\theta)(d-s)}{1+p}}\right)^{1+p}\delta^{s-(d-1)}|U|^{d-s}|U|^{s} + \eta_{d-1}\delta^{s-(d-1)}|U|^{d-s-1}|U|^{s}\\
    &(\text{since }\delta^{-\frac{1-(1-\theta)(d-s)}{1+p}}\geq1)\\
    &\leq \frac{2^{1+p}\eta_{d-1}}{p}\delta^{-\theta(d-s)}\delta^{\theta(d-s)}|U|^{s} + \eta_{d-1}\delta^{s-(d-1)}\delta^{\theta(d-s-1)}|U|^{s}\\
    &= \frac{2^{1+p}\eta_{d-1}}{p}|U|^{s} + \eta_{d-1}\delta^{(1-\theta)\left(\frac{\theta(1-p(d-1))}{dp+\theta(1-p(d-1))}\right)}|U|^{s}\\
    &\leq \left( \frac{2^{1+p}\eta_{d-1}}{p} + \eta_{d-1}\right)|U|^{s}\\
    &(\text{since }(1-\theta)\geq0\text{ and } \frac{\theta(1-p(d-1))}{dp+\theta(1-p(d-1))}\geq0).\\
\end{align*}

Hence, by the mass distribution principle for the intermediate dimensions \cite[Proposition 2.2]{intermediate}, $\lodim_\theta\left(C_p^d\right)\geq s=\frac{dp(d-1)+d\theta(1-p(d-1))}{dp+\theta(1-p(d-1))}$.

For an upper bound, let $0<\delta<1$, $\frac{dp(d-1)+d\theta(1-p(d-1))}{dp+\theta(1-p(d-1))}<s<d$, and write $M=\left\lceil\delta^{-\frac{1-(1-\theta)(d-s)}{1+p}}\right\rceil$. We can cover $\bigcup_{i=M}^{\infty} \mathcal{S}_{1/i^p}^{d-1} \subseteq B_{\mathbb{R}^d}\left(0,\frac{1}{M^p}\right)$ with a grid of at most $\left(\frac{2\sqrt{d}}{M^p \delta^\theta}+1\right)^{d}= \sum_{k=0}^{d} \binom{d}{k}\left(\frac{2\sqrt{d}}{M^p \delta^\theta}\right)^k$ $d$-cubes of side length $\delta^\theta/\sqrt{d}$.

For each $1\leq i\leq M-1$, we have $\frac{1}{i^p}\geq \frac{1}{(M-1)^p} \geq \delta^{\frac{p}{1+p}(1-(1-\theta)(d-s))}>\delta$, so we can cover $\mathcal{S}_{1/i^p}^{d-1}$ with at most $\xi_d \left(\frac{1}{i^p \delta}\right)^{d-1}$ sets of diameter $\delta$ (where $\xi_d$ is a constant depending only on $d$). Thus, we can cover $\bigcup_{i=1}^{M-1} \mathcal{S}_{1/i^p}^{d-1}$ with at most $\sum_{i=1}^{M-1} \xi_d \left(\frac{1}{i^p \delta}\right)^{d-1}$ sets of diameter $\delta$. This can be simplified as:
\begin{align*}
    \sum_{i=1}^{M-1} \xi_d \left(\frac{1}{i^p \delta}\right)^{d-1} = \frac{\xi_d}{\delta^{d-1}} \sum_{i=1}^{M-1} \left(\frac{1}{i^{p(d-1)} }\right)
    &\leq \frac{\xi_d}{\delta^{d-1}} \int_{x=0}^{M-1} \frac{1}{x^{p(d-1)} } dx\\
    &= \frac{\xi_d}{1-p(d-1)} \frac{(M-1)^{1-p(d-1)}}{\delta^{d-1}}.
\end{align*}

This gives us a cover $\mathcal{U}$ of $C_p^d$ such that $\delta\leq|U|\leq\delta^\theta$ for all $U\in\mathcal{U}$. Summing over the sets in this cover, we get:
\begin{align*}
    \sum_{U\in\mathcal{U}} |U|^s &\leq \delta^{\theta s}\sum_{k=0}^{d} \binom{d}{k}\left(\frac{2\sqrt{d}}{M^p \delta^\theta}\right)^k + \delta^s \frac{\xi_d}{1-p(d-1)} \frac{(M-1)^{1-p(d-1)}}{\delta^{d-1}}\\
    &\leq \sum_{k=0}^{d} \left[ \binom{d}{k}\left(2\sqrt{d}\right)^k \delta^{\frac{1-(1-\theta)(d-s)}{1+p}pk} \delta^{\theta s - \theta k} \right] + \frac{\xi_d}{1-p(d-1)} \frac{\delta^{-\frac{1-(1-\theta)(d-s)}{1+p}(1-p(d-1))}\delta^s}{\delta^{d-1}}\\
    &= \sum_{k=0}^{d} \left[ \binom{d}{k}\left(2\sqrt{d}\right)^k \delta^{\frac{k(p-p(1-\theta)(d-s)- \theta  - \theta p) +\theta s + \theta p s }{1+p}} \right] \\
    &+ \frac{\xi_d}{1-p(d-1)} \delta^{\frac{p(d-1)-p(d-1)(1-\theta)(d-s)-1+(1-\theta)(d-s)+s+sp-d-dp+1+p}{1+p}}\\
    &= \sum_{k=0}^{d} \left[ \binom{d}{k}\left(2\sqrt{d}\right)^k \delta^{\frac{\frac{k}{d}(dp-dp(1-\theta)(d-s)-\theta d-\theta dp +\theta s + \theta ps) - \frac{k}{d}(\theta s + \theta ps) +\theta s + \theta p s }{1+p}} \right] \\
    &+ \frac{\xi_d}{1-p(d-1)} \delta^{\frac{-p(d-1)(1-\theta)(d-s)+(1-\theta)(d-s)+s+sp-d}{1+p}}\\
    &= \sum_{k=0}^{d} \left[ \binom{d}{k}\left(2\sqrt{d}\right)^k \delta^{\frac{\frac{k}{d}\left(s[dp+\theta(1-p(d-1))]-[dp(d-1)+d\theta(1-p(d-1))]\right) + \left(1- \frac{k}{d}\right)(\theta s + \theta ps)}{1+p}} \right] \\
    &+ \frac{\xi_d}{1-p(d-1)} \delta^{\frac{s[dp+\theta(1-p(d-1))]-[dp(d-1)+d\theta(1-p(d-1))]}{1+p}},\\
\end{align*}
which converges to $0$ as $\delta\to0^+$, since our assumption that $\frac{dp(d-1)+d\theta(1-p(d-1))}{dp+\theta(1-p(d-1))}<s<d$ and the fact that $0\leq\frac{k}{d}\leq1$ within the summation ensure that the exponents of $\delta$ are all positive. This tells us that for any given $\varepsilon>0$, we can find some $\delta_0>0$ such that for all $0<\delta<\delta_0$, there is a cover $\mathcal{U}$ of $C_p^d$ such that $\delta\leq|U|\leq\delta^\theta$ for all $U\in\mathcal{U}$ and $\sum_{U\in\mathcal{U}} |U|^s\leq\varepsilon$. Hence, $\updim_\theta\left(C_p^d\right)\leq s$ for all $\frac{dp(d-1)+d\theta(1-p(d-1))}{dp+\theta(1-p(d-1))}<s<d$, which implies that $\updim_\theta\left(C_p^d\right)\leq\frac{dp(d-1)+d\theta(1-p(d-1))}{dp+\theta(1-p(d-1))}$.

\end{proof}

Consider the polynomial spiral
\begin{equation}
    S_p:=\left\{\left(\frac{1}{t^p}\sin(\pi t),\frac{1}{t^p}\cos(\pi t)\right):t\geq1\right\}\subseteq\mathbb{R}^2,
\end{equation}
where $p>0$ is the polynomial winding rate. Figure \ref{spiral} illustrates $S_1$.

\begin{figure}
    \centering
    \begin{subfigure}[h]{7cm}
        \centering
        \includegraphics[width=5cm]{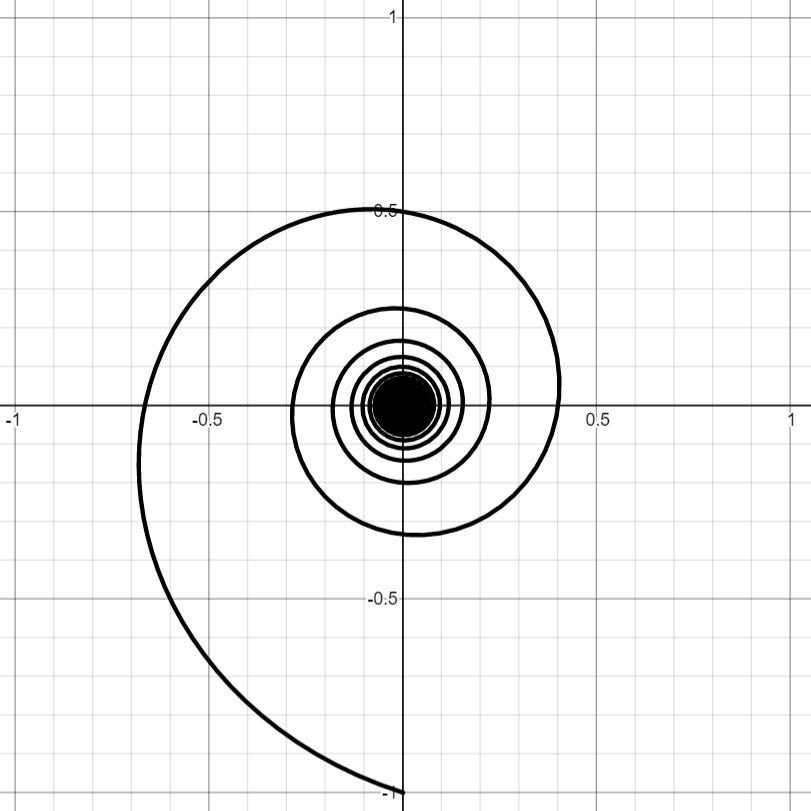} 
        \caption{The set $S_1\subseteq\mathbb{R}^2$.} \label{spiral}
    \end{subfigure}
    \begin{subfigure}[h]{7cm}
        \centering
        \includegraphics[width=5cm]{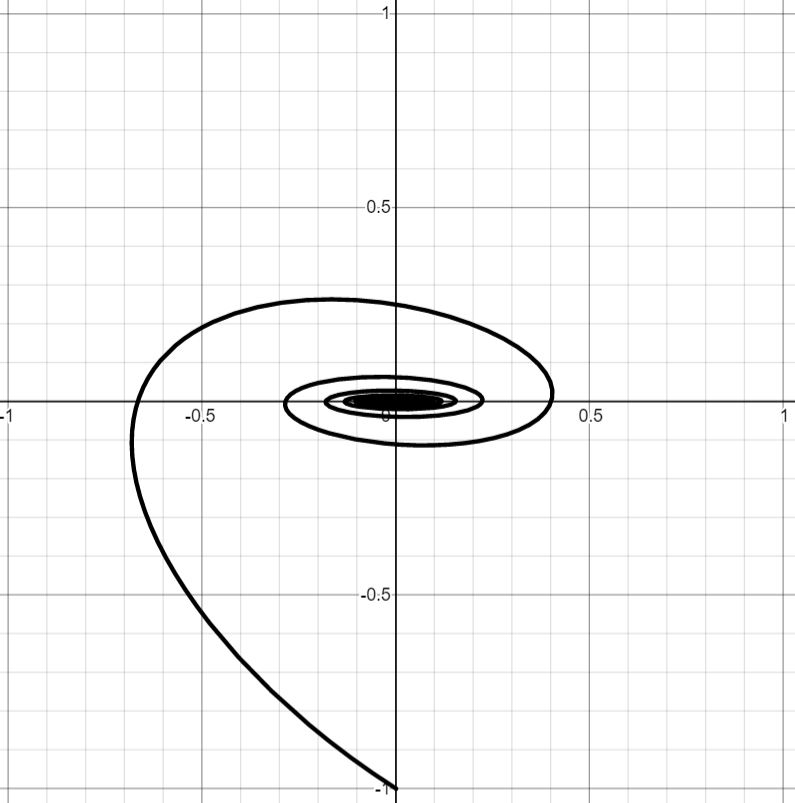}
        \caption{The set $S_{1,2}\subseteq\mathbb{R}^2$.} \label{spiral2}
    \end{subfigure}
    \caption{Illustrations of (\subref{spiral}) a polynomial spiral and (\subref{spiral2}) an elliptical polynomial spiral.}
\end{figure}

Since the $i$th circular arc of the spiral has length between $\frac{2\pi}{i^p}$ and $\frac{2\pi}{(i+1)^p}$ and moreover is contained in $\left\{x\in\mathbb{R}^d:\frac{1}{(i+1)^p}\leq|x|\leq\frac{1}{i^p}\right\}$, the proof of Theorem \ref{concentric1} can be modified easily to give
\begin{equation}\label{spiralformula}
    \lodim_\theta \left(S_p\right) = \updim_\theta \left(S_p\right) =
    \begin{cases}
        1+\frac{\theta(1-p)}{2p+\theta(1-p)} &\text{if } 0<p<1\\
        1 &\text{if } p\geq1
    \end{cases}
\end{equation}
for all $\theta\in[0,1]$. Thus, the intermediate dimensions of this relatively simple set displays non-trivial interpolation and continuity at zero.

The formula (\ref{spiralformula}) was also recently obtained independently by Burrell, Falconer, and Fraser \cite{elliptical} as a corollary to their calculation of the intermediate dimensions of elliptical polynomial spirals in $\mathbb{R}^2$ (for example, as shown in Figure \ref{spiral2}). The special case of $\theta=1$, which is simply the box dimension, has been calculated previously, for example, in 1991 by Vassilicos and Hunt \cite[Equation (A 22)]{kolmogorov}.

\section{Sets of concentric spheres generated by more general sequences}

It is natural to move on to sets of concentric spheres in $\mathbb{R}^d$ with radii from more general sequences in $\mathbb{R}$. For $(a_n)_{n\in\mathbb{N}}$ a decreasing sequence in $\mathbb{R}$ converging to $0$, write $\{a_n\}=\{a_n:n\in\mathbb{N}\}\subseteq\mathbb{R}$ and let
\begin{equation}
    C^d\left(\{a_n\}\right):=\left\{x\in\mathbb{R}^d:|x|\in \{a_n\}\right\}.
\end{equation}

This gives us a wide class of sets in $\mathbb{R}^d$ which are interesting to study from the perspective of the intermediate dimensions. We will approach this in two separate ways: first, by comparing $(a_n)_{n\in\mathbb{N}}$ to our benchmark $F_p$ sets; and second, by bounding $\dim_\theta\left(C^d\left(\{a_n\}\right)\right)$ in terms of $\dim_\theta\left(\{a_n\}\right)$ without any other information about the sequence itself.

The following facts can be observed simply by modifying the proof of Theorem \ref{concentric1}.

\begin{proposition}\label{concentriccor}
Let $(a_n)_{n\in\mathbb{N}}$ be a decreasing sequence in $\mathbb{R}$ converging to $0$, let $d\in\mathbb{N}$, and let $A_{p,n}$ be defined by
\begin{equation}\label{apndef}
    A_{p,n}\left((a_i)_{i\in\mathbb{N}}\right):=\#\left\{\left(\frac{1}{(k+1)^p},\frac{1}{k^p}\right]:1\leq k\leq n , \left(\frac{1}{(k+1)^p},\frac{1}{k^p}\right]\cap\left\{a_i:i\in\mathbb{N}\right\}\not=\emptyset\right\}.
\end{equation}

\begin{enumerate}
    \item For $p>0$, if $\limsup_{n\to\infty} \left(a_n n^p\right)<1$, then $\updim_\theta\left(C^d\left(\{a_n\}\right)\right) \leq \updim_\theta \left(C_p^d\right)$.
    \item For $p>0$, if $\liminf_{n\to\infty}\frac{A_{p,n}\left((a_i)_{i\in\mathbb{N}}\right)}{n}>0$, then $\lodim_\theta\left(C^d\left(\{a_n\}\right)\right) \geq \lodim_\theta \left(C_p^d\right)$.
\end{enumerate}

\end{proposition}

\begin{proof}
We will not go through the full calculations again, but instead state the modifications to the proof of Theorem \ref{concentric1}.

\begin{enumerate}
    \item If $\limsup_{n\to\infty} \left(a_n n^p\right)<1$, then for all large enough $n\in\mathbb{N}$, the $n$th sphere of $C^d\left(\{a_n\}\right)$ is a scaled-down version (shrunk towards the origin) of the $n$th sphere of $C_p^d$. We can temporarily remove the finite number of spheres of $C^d\left(\{a_n\}\right)$ which are strictly larger than the corresponding spheres of $C_p^d$ due to finite stability of the upper intermediate dimensions, and then modify the efficient cover constructed for $C_p^d$ in the proof of Theorem \ref{concentric1} to become a permissible cover for $C^d\left(\{a_n\}\right)$ without increasing the sum $\sum_{U\in\mathcal{U}} |U|^s$.
    \item $\liminf_{n\to\infty}\frac{A_{p,n}\left((a_i)_{i\in\mathbb{N}}\right)}{n}=\alpha>0$ implies that for all large enough $N\in\mathbb{N}$, at least $\frac{\alpha}{2}$ of the gaps between consecutive spheres (the first N of them) contains a sphere of $C^d\left(\{a_n\}\right)$. We can apply the same $(d-1)$-dimensional Hausdorff measure on one (arbitrarily chosen) sphere in each such gap, scaled by the factor $\delta^{s-(d-1)}$, and the mass distribution principle \cite[Proposition 2.2]{intermediate} will give the result.
\end{enumerate}

For $d=1$, the one should instead modify the proof of Proposition 3.1 in \cite{intermediate}, which calculated the intermediate dimensions of $F_p$.
\end{proof}

Proposition \ref{concentriccor} tells us that the intermediate dimensions of $C^d\left(\left\{\frac{1}{q^n}\right\}\right)$ (given some $q\geq1$) will be $d-1$ for all $\theta\in[0,1]$, whereas the intermediate dimensions of $C^d\left(\left\{\frac{1}{\log n}\right\}\right)$ will be $d$ for $\theta\in(0,1]$ and $d-1$ for $\theta=0$ (in particular, discontinuous at zero).

Given some $(a_n)_{n\in\mathbb{N}}$ a decreasing sequence in $\mathbb{R}$ converging to $0$, it is not straightforward to derive the intermediate dimensions of $C^d\left(\{a_n\}\right)$ using only intermediate dimensions of $\{a_n\}$. We can, however, provide some loose bounds:
\begin{equation}\label{pointsbounds}
    \min\{d-1,d\updim_\theta\left(\{a_n\}\right)\} \leq \updim_\theta\left(C_d\left(\{a_n\}\right)\right) \leq d-1+\updim_\theta\left(\{a_n\}\right) \leq d,
\end{equation}
and similarly for the lower intermediate dimensions.
We prove the upper bound first (Proposition \ref{upbound}), followed by the lower bound (Theorem \ref{lobound}).

\begin{proposition}\label{upbound}
Let $(a_n)_{n\in\mathbb{N}}$ be a decreasing sequence in $\mathbb{R}$ converging to $0$. Then for all $d\in\mathbb{N}$ and any $\varepsilon>0$, we have
\begin{equation}
    \updim_\theta\left(C^d\left(\{a_n\}\right)\right) \leq d-1 + \updim_\theta \left(\{a_n\}\right),
\end{equation}
and
\begin{equation}
    \lodim_\theta\left(C^d\left(\{a_n\}\right)\right) \leq d-1 + \lodim_\theta \left(\{a_n\}\right).
\end{equation}
\end{proposition}

\begin{proof}
Consider $\widetilde{C^d}\left(\{a_n\}\right) := \left\{x\in\mathbb{R}^d:|x|-1\in \{a_n\}\right\}$. That is, the set of concentric spheres in $\mathbb{R}^d$ with radii coming from $\{a_n\}+1=\{a_n+1:n\in\mathbb{N}\}$, resulting in the set of limit points being the unit $(d-1)$-sphere in $\mathbb{R}^d$. We can break $\widetilde{C^d}\left(\{a_n\}\right)$ into $2^d$ pieces based on quadrants, and each piece will be bi-Lipschitz equivalent to $\{a_n\}\times[0,1]\times...\times[0,1]\subseteq\mathbb{R}^d$, and therefore have the same intermediate dimensions \cite[Lemma 3.1]{interpolating}.

The product formula for the intermediate dimensions \cite[Proposition 2.5]{intermediate} gives
\begin{equation*}
    \lodim_\theta(E)+\lodim_\theta(F) \leq \lodim_\theta(E\times F) \leq \lodim_\theta(E)+\updim_B(F)
\end{equation*}
and
\begin{equation*}
    \updim_\theta(E)+\lodim_\theta(F) \leq \updim_\theta(E\times F) \leq \updim_\theta(E)+\updim_B(F).
\end{equation*}
Therefore, $\lodim_\theta\left(\{a_n\}\times[0,1]\times...\times[0,1]\right) = \lodim_\theta \left(\{a_n\}\right) + d-1$, which then implies that
\begin{equation}\label{withgap}
    \lodim_\theta\left(\widetilde{C^d}\left(\{a_n\}\right)\right) = d-1 + \lodim_\theta \left(\{a_n\}\right),
\end{equation}
and similarly for the upper intermediate dimensions. Even though the lower intermediate dimensions are not countably stable in general, (\ref{withgap}) holds because $\widetilde{C^d}\left(\{a_n\}\right)$ was partitioned into a finite number of congruent pieces.

The mapping $T:\{x\in\mathbb{R}^d:|x|\geq1\}\to\mathbb{R}^d$ given by $T(x)=\frac{|x|-1}{|x|}x$ is non-expansive, and therefore Lipschitz. Since the intermediate dimensions do not increase under Lipschitz mappings \cite[Theorem 3.1]{brownian}, the intermediate dimensions of $C^d\left(\{a_n\}\right)$ are bounded above by the intermediate dimensions of $\widetilde{C^d}\left(\{a_n\}\right)$.
\end{proof}

\begin{thm}\label{lobound}
Let $(a_n)_{n\in\mathbb{N}}$ be a decreasing sequence in $\mathbb{R}$ converging to $0$. Then for $d\geq2$ and $\theta\in(0,1]$, we have
\begin{equation}
    \lodim_\theta\left(C^d\left(\{a_n\}\right)\right) \geq d \lodim_\theta\left(\{a_n\}\right)
\end{equation}
and
\begin{equation}
    \updim_\theta\left(C^d\left(\{a_n\}\right)\right) \geq d \updim_\theta\left(\{a_n\}\right).
\end{equation}
\end{thm}

\begin{proof}
Let $0<s<\lodim_\theta\left(\{a_n\}\right)$ (if $\lodim_\theta\left(\{a_n\}\right)=0$ then the result is trivial). Assume that $\{a_n\}\subseteq[0,1]$, otherwise just scale it down, noting that the intermediate dimensions are scaling-invariant. Since $\overline{\{a_n\}}=\{a_n\}\cup\{0\}$ is a compact set, so we can apply the Frostman-type lemma for the intermediate dimensions \cite[Proposition 2.3]{intermediate} to obtain a constant $c>0$ such that for all $\delta\in(0,1)$, there is a Borel probability measure $\mu_\delta$ supported on $\overline{\{a_n\}}\subseteq[0,1]$ such that for all $x\in\mathbb{R}$ and $\delta^{\frac{1}{\theta}}\leq r\leq\delta$, $\mu_\delta(B(x,r))\leq cr^s$.

Set $\delta_0=\min\{\frac{1}{2},(2c)^{\frac{1}{1-s}},\left(\frac{1}{4c}\right)^{\frac{1}{s}}\}$. Then for all $0<\delta<\delta_0$, modify the measures from the Frostman-type lemma by $\widetilde{\mu_\delta}=\mu_\delta \restriction_{\left(\left\lfloor\frac{1}{2c\delta^{s}}\right\rfloor\delta,1\right]}$. Using our upper bound on $\delta$, we have $\frac{1}{2c\delta^s}>\frac{1}{2c\delta_0^s}\geq\frac{4c}{2c}=2$, so the origin does not carry any mass, and also $\left\lfloor\frac{1}{2c\delta^{s}}\right\rfloor \geq \frac{1}{4c\delta^{s}}$. Since $\delta<(2c)^{\frac{1}{1-s}}$,
\begin{equation*}
    \left\lfloor\frac{1}{2c\delta^{s}}\right\rfloor\delta \leq \frac{1}{2c\delta^{s}}\delta = \frac{1}{2c}\delta^{1-s} < 1.
\end{equation*}
We also have
\begin{equation*}
    \mu_\delta\left(\left[0,\left\lfloor\frac{1}{2c\delta^{s}}\right\rfloor\delta\right]\right) \leq \sum_{i=1}^{\left\lfloor\frac{1}{2c\delta^{s}}\right\rfloor}\mu_\delta\left(\left[(i-1)\delta,i\delta\right]\right)
    \leq \sum_{i=1}^{\left\lfloor\frac{1}{2c\delta^{s}}\right\rfloor}c\delta^s
    = \left\lfloor\frac{1}{2c\delta^{s}}\right\rfloor c\delta^s
    \leq \frac{1}{2c\delta^{s}} c\delta^s
    = \frac{1}{2},
\end{equation*}
so $\widetilde{\mu_\delta}(\{a_n\})\geq\frac{1}{2}$ for all $0<\delta<\delta_0$. Furthermore, $\widetilde{\mu_\delta}$ is supported on $\{a_n\} \cap \left(\left\lfloor\frac{1}{2c\delta^{s}}\right\rfloor\delta,1\right]$. For each of the finite number of point masses, denote their positions by $x_i\in\{a_n\} \cap \left(\left\lfloor\frac{1}{2c\delta^{s}}\right\rfloor\delta,1\right]$ and their respective masses by $m_i\in(0,1]$ (for $1\leq i \leq n_\delta$). Clearly, $x_i>\left\lfloor\frac{1}{2c\delta^{s}}\right\rfloor\delta \geq \frac{1}{4c}\delta^{1-s}$ for all $1\leq i \leq n_\delta$.

We now construct appropriate measures to apply the mass distribution principle. First, we define the normalised $(d-1)$-spherical measure of radius $x>0$ by
\begin{equation}\label{sphericalmeasure}
    \sigma_{x}^{d-1}:=\frac{1}{\mathcal{H}^{d-1}\left(\mathcal{S}_{x}^{d-1}\right)}\mathcal{H}^{d-1}\restriction_{\mathcal{S}_{x}^{d-1}}.
\end{equation}
Now, we define
\begin{equation}
    \lambda_\delta:=\sum_{i=1}^{n_\delta}m_i\sigma_{x_i}^{d-1}.
\end{equation}
Then $\lambda_\delta$ is a Borel measure supported on $C^d(\{a_n\})$. The total mass carried is:
\begin{align*}
    \lambda_\delta\left(C^d(\{a_n\})\right) &=\sum_{i=1}^{n_\delta}m_i\sigma_{x_i}^{d-1}\left(C^d(\{a_n\})\right)\\
    &= \sum_{i=1}^{n_\delta}m_i\frac{1}{\mathcal{H}^{d-1}\left(\mathcal{S}_{x_i}^{d-1}\right)}\mathcal{H}^{d-1}\left(\mathcal{S}_{x_i}^{d-1}\right)\\
    &=\sum_{i=1}^{n_\delta}m_i
    \geq \frac{1}{2}.
\end{align*}

Now, suppose that $U\subseteq\mathbb{R}^d$ is such that $\delta^{\frac{1}{\theta}} \leq |U| \leq \delta$. Let $I\subseteq\{1,...,n_\delta\}$ be the finite index set such that $U$ intersects $\mathcal{S}_{x_i}^{d-1}$ for all $i\in I$. Then there is a set $V\subseteq\mathbb{R}$ such that $|V|\leq|U|$ and $x_i\in V$ for all $i\in I$. Now, we have:
\begin{equation}\label{propineq}
    \sum_{i\in I} \lambda_\delta\left(\mathcal{S}_{x_i}^{d-1}\right) = \sum_{i\in I} m_i
    = \sum_{i\in I} \widetilde{\mu_\delta}(\{x_i\})
    \leq \widetilde{\mu_\delta}(V)
    \leq \mu_\delta(V)
    \leq c|V|^s
    \leq c|U|^s.
\end{equation}

Therefore, the mass carried by $U$ is:
\begin{align*}
    \lambda_\delta\left(U\right) = \sum_{i\in I} \lambda_\delta\left(U\cap\mathcal{S}_{x_i}^{d-1}\right)
    &= \sum_{i\in I} m_i \sigma_{x_i}^{d-1}\left(U\right)\\
    &= \sum_{i\in I} m_i \frac{1}{\mathcal{H}^{d-1}\left(\mathcal{S}_{x_i}^{d-1}\right)}\mathcal{H}^{d-1}\restriction_{\mathcal{S}_{x_i}^{d-1}} \left(U\right)\\
    &= \sum_{i\in I} m_i \frac{1}{\eta_{d-1}x_i^{d-1}}\mathcal{H}^{d-1} \left(U\cap\mathcal{S}_{x}^{d-1}\right)\\
    &\leq \sum_{i\in I} m_i \frac{1}{\eta_{d-1}x_i^{d-1}}\eta_{d-1}|U|^{d-1}\\
    &\leq \sum_{i\in I} m_i \frac{1}{\left(\frac{1}{4c}\delta^{1-s}\right)^{d-1}}|U|^{d-1}\\
    &= (4c)^{d-1}\delta^{-(1-s)(d-1)}|U|^{d-1} \sum_{i\in I} m_i\\
    &\leq (4c)^{d-1}\delta^{-(1-s)(d-1)}|U|^{d-1} c|U|^s\\
    &(\text{by }(\ref{propineq}))\\
    &= 4^{d-1} c^{d}\delta^{-(1-s)(d-1)}|U|^{s+d-1-ds} |U|^{ds}\\
    &\leq 4^{d-1} c^{d}\delta^{-(1-s)(d-1)}\delta^{s+d-1-ds} |U|^{ds}\\
    &= 4^{d-1} c^{d} |U|^{ds}.\\
\end{align*}
Thus, by the mass distribution principle \cite[Proposition 2.2]{intermediate}, for all $0<s<\lodim_\theta\left(\{a_n\}\right)$, we have $\lodim_\theta\left(C^d(\{a_n\})\right) \geq ds$. Hence, we conclude that $\lodim_\theta\left(C^d(\{a_n\})\right) \geq d\lodim_\theta\left(\{a_n\}\right)$.

For the statement on the upper intermediate dimensions, we follow the same steps as above, but for a specific sequence $\delta\to0^+$ instead of all $0<\delta<\delta_0$.
\end{proof}

Proposition \ref{upbound} and Theorem \ref{lobound} give us the following corollary relating to continuity at zero:

\begin{corollary}
Let $(a_n)_{n\in\mathbb{N}}$ be a decreasing sequence in $\mathbb{R}$ converging to $0$ and let $d\geq 2$.

\begin{enumerate}
    \item If the intermediate dimensions of $\{a_n\}$ are continuous at zero, then so are the intermediate dimensions of $C^d\left(\{a_n\}\right)$.
    \item For $\theta\in(0,1]$, if $\{a_n\}$ has full $\theta$-intermediate dimension (that is, equal to unity), then $C^d\left(\{a_n\}\right)$ will have full $\theta$-intermediate dimension (equal to $d$).
\end{enumerate} 
\end{corollary}

The bounds in (\ref{pointsbounds}) leave a significant gap, and it is not yet known whether the intermediate dimensions of $C^d\left(\{a_n\}\right)$ can be calculated using only the intermediate dimensions of $\{a_n\}$. In fact, our bounds (\ref{pointsbounds}) do not preclude the case of the intermediate dimensions of $\{a_n\}$ being discontinuous at zero and the intermediate dimensions of $C^d\left(\{a_n\}\right)$ being continuous at zero.

\section{Isolated points on concentric spheres}

By construction, the intermediate dimensions of $C_p^d=\bigcup_{i\in\mathbb{N}}\mathcal{S}_{1/i^p}^{d-1}\subseteq\mathbb{R}^d$ have to be between $d-1$ and $d$. Furthermore, as we move to higher-dimensional ambient Euclidean space, the range of values of the parameter $p>0$ for which non-trivial interpolation is observed shrinks. Now, we consider countable subsets of $C_p^d$ since they can display interpolation over a wider range. Due to countable stability, the Hausdorff dimension of these sets will be zero. In order to keep the discussion focused around the structure of the $C_p^d$ sets, we require that the only limit point is the origin. It is straightforward to show that the condition that the origin is the only limit point is equivalent to the condition that there are only a finite number of points on each sphere, and also equivalent to the condition that each point is isolated. Such sets in $\mathbb{R}^2$ were studied by Mendivil and Saunders \cite{measurability} as examples of sets with box dimension $d\in(0,2)$ and arbitrary positive Minkowski content.

For $i\in\mathbb{N}$, we denote the (finite) number of points contained in the $i$th sphere, $\mathcal{S}_{1/i^p}^{d-1}$, by $b_i$. These sets (for general $b_i$) are difficult to work with, and we will only provide a sufficient condition that ensures continuity at zero.

Such sets would only be interesting if there can be non-trivial interpolation. That is, if $\dim_B(E)>0$. This is indeed the case, as long as there are non-empty spheres at regular intervals. Let $E\subseteq C_p^d$ be a set of isolated points on concentric spheres. If $\liminf_{n\to\infty}\frac{\#\{1\leq i\leq n : b_i>0\}}{n}>0$, then $\lodim_\theta(E)\geq\frac{\theta}{p+\theta}$ for all $\theta\in[0,1]$. Similarly, if $\limsup_{n\to\infty}\frac{\#\{1\leq i\leq n : b_i>0\}}{n}>0$, then $\updim_\theta(E)\geq\frac{\theta}{p+\theta}$ for all $\theta\in[0,1]$.

On the other hand, it is also possible for such sets to have zero box dimension. For example, the set $\left\{\frac{1}{2^{np}}:n\in\mathbb{N}\right\}\times\{0\}\times...\times\{0\}\subseteq C_p^d$.

We now provide a sufficient condition on the growth of $b_i$ that ensures that the intermediate dimensions will be continuous at zero.

\begin{proposition}\label{points}
For $d\geq2$ and $p>0$, let $E\subseteq C_p^d$ be a set of isolated points on concentric spheres such that the number of points contained in $\mathcal{S}_{1/i^p}^{d-1}$, $b_i$, for $i\in\mathbb{N}$, satisfies $\limsup_{n\to\infty}\frac{\log\left(\sum_{i=1}^{n}b_i\right)}{\log(n)}<\infty$. Then the intermediate dimensions of $E$ are continuous at zero.

If we only have that $\liminf_{n\to\infty}\frac{\log\left(\sum_{i=1}^{n}b_i\right)}{\log(n)}<\infty$, then we can only conclude that the lower intermediate dimensions of $E$ are continuous at zero.
\end{proposition}

\begin{proof}
Suppose that $\limsup_{n\to\infty}\frac{\log\left(\sum_{i=1}^{n}b_i\right)}{\log(n)}<l<\infty$. Then:
\begin{gather*}
    \limsup_{n\to\infty}\frac{\log\left(\sum_{i=1}^{n}b_i\right)}{\log(n)} < l\\
    \limsup_{n\to\infty} \log(n)\left(\frac{\log\left(\sum_{i=1}^{n}b_i\right)}{\log(n)} -l\right) < 0\\
    \limsup_{n\to\infty} \log\left(\sum_{i=1}^{n}b_i\right) -\log(n^l) < 0\\
    \limsup_{n\to\infty} \left(\frac{\sum_{i=1}^{n}b_i}{n^l}\right) < 1,\\
\end{gather*}
so there is some $N\in\mathbb{N}$ such that for all $n\geq N$, we have $\sum_{i=1}^{n}b_i<n^l$.

Let $s>\frac{\theta ld}{pd+\theta l}$, let $\delta_0=(N+1)^{-\frac{l+pd}{s-\theta s+ \theta d}}$, and let $0<\delta<\delta_0$ be given. Take $M=\left\lceil\delta^{-\frac{s-\theta s+\theta d}{l+pd}}\right\rceil$. Then $M-1\geq N$.

We can cover $B\left(0,\frac{1}{M^p}\right)$ by at most $\left(\frac{2\sqrt{d}}{M^p \delta^\theta}+1\right)^d$ $d$-cubes of side length $\delta^\theta/\sqrt{d}$ (and thus diameter $\delta^\theta$). By the binomial theorem, $\left(\frac{2\sqrt{d}}{M^p \delta^\theta}+1\right)^{d}= \sum_{k=0}^{d} \binom{d}{k}\left(\frac{2\sqrt{d}}{M^p \delta^\theta}\right)^k$.

We also know that $E\setminus B\left(0,\frac{1}{M^p}\right)$ comprises $\sum_{i=1}^{M-1}b_i$ points, and can therefore be covered by $\sum_{i=1}^{M-1}b_i$ sets of diameter $\delta$. Hence, we get a cover $\mathcal{U}$ of $E$ such that $\delta\leq|U|\leq\delta^\theta$ for all $U\in\mathcal{U}$. Summing over all the sets in this cover:
\begin{align*}
    \sum_{U\in\mathcal{U}}|U|^s &\leq \delta^{\theta s} \sum_{k=0}^{d} \binom{d}{k}\left(\frac{2\sqrt{d}}{M^p \delta^\theta}\right)^k + \delta^s \sum_{i=1}^{M-1}b_i\\
    &\leq \sum_{k=0}^{d} \left[\binom{d}{k}\left(2\sqrt{d}\right)^k M^{-kp} \delta^{\theta s - \theta k} \right] + \delta^s (M-1)^l\\
    &\leq \sum_{k=0}^{d} \left[\binom{d}{k}\left(2\sqrt{d}\right)^k \delta^{\frac{s-\theta s+\theta d}{l+pd}kp+\theta s-\theta k} \right] + \delta^{s-\frac{s-\theta s+\theta d}{l+pd}l}\\
    &= \sum_{k=0}^{d} \left[\binom{d}{k}\left(2\sqrt{d}\right)^k \delta^{\frac{kps-\theta kps+\theta kpd+\theta sl + \theta spd - \theta kl - \theta kpd}{l+pd}} \right] + \delta^{\frac{sl+spd-sl+\theta sl-\theta ld}{l+pd}}\\
    &= \sum_{k=0}^{d} \left[\binom{d}{k}\left(2\sqrt{d}\right)^k \delta^{\frac{\frac{k}{d}\left(s[pd+\theta l] -[\theta ld]\right)+\left(1-\frac{k}{d}\right)\theta sl + \theta sp(d-k)}{l+pd}} \right] + \delta^{\frac{s[pd+\theta l]-[\theta ld]}{l+pd}},\\
\end{align*}
which converges to $0$ as $\delta\to0^+$ under our assumption that $s>\frac{\theta ld}{pd+\theta l} \geq0$, and noting that $0\leq\frac{k}{d}\leq1$ within the summation.

This allows us to conclude that $\updim_\theta(E)\leq\frac{\theta ld}{pd+\theta l}$, which decreases to $0$ as $\theta\to0$.

If we only have that $\liminf_{n\to\infty}\frac{\log\left(\sum_{i=1}^{n}b_i\right)}{\log(n)}<\infty$, then we only have that $\liminf_{n\to\infty} \left(\frac{\sum_{i=1}^{n}b_i}{n^l}\right) < 1$, which gives the weaker observation that there is some $M_j\to\infty$ such that for all $j\in\mathbb{N}$, we have $\sum_{i=1}^{M_j-1}b_i<(M_j-1)^l$. Letting $s>\frac{\theta ld}{pd+\theta l}$ and setting $\delta_j=M_j^{-\frac{l+pd}{s-\theta s+ \theta d}}$ for $j\in\mathbb{N}$, we get a sequence $\delta_j\to0$. Now using the same cover as before, for each $\delta_j$, there is a cover $\mathcal{U}_j$ of $E$ with $\delta_j\leq|U|\leq\delta_j^\theta$ such that $\sum_{U\in\mathcal{U}}|U|^s\to0$ as $j\to\infty$. This allows us to conclude that $\lodim_\theta(E)\leq s$ for all $s>\frac{\theta ld}{pd+\theta l}$, and so $\lodim_\theta(E)\leq\frac{\theta ld}{pd+\theta l}$.
\end{proof}

For the condition provided in Proposition \ref{points} to be useful, it must be possible for the intermediate dimensions of such sets to be discontinuous at zero. We show this with an example in $\mathbb{R}^2$.
\begin{example}
Let $p>0$ and $E\subseteq C_p^2$ be the set comprising $b_i=2^i$ points evenly spread out on $\mathcal{S}_{1/i^p}^{1}$ (for $i\in\mathbb{N}$). Then the intermediate dimensions of $E$ are discontinuous at zero.
\end{example}

\begin{proof}
Let $\theta\in(0,1]$ and write $\gamma=\frac{\theta}{4p}>0$. There is some $\delta_\gamma>0$ such that for all $0<\delta<\delta_\gamma$, we have $\frac{1}{2^{\delta^{-\gamma}}}<\delta^{\frac{1}{2}}$. Now let $0<\delta<\min\left\{\delta_\gamma,1\right\}$ be given, and set $M=\left\lceil\delta^{-\gamma}\right\rceil$. Define $\mu_\delta:=\frac{1}{2^M}\mathcal{H}^0\restriction_{\mathcal{S}_{1/M^p}^{1}}$. Since there are $2^M$ points on $\mathcal{S}_{1/M^p}^{1}$ and each point has mass $\frac{1}{2^M}$, $\mu_\delta$ is a Borel probability measure supported on $E$.

Suppose that $U\subseteq\mathbb{R}^2$ is such that $\delta\leq|U|\leq\delta^\theta$. Then $U$ intersects at most $\frac{\pi|U|}{2\pi/{M^p}}2^M+1$ points carrying mass. This implies:
\begin{align*}
    \mu_\delta(U) \leq \left(\frac{\pi|U|}{2\pi/{M^p}}2^M+1\right)\frac{1}{2^M}
    &= \frac{1}{2}|U|M^p + \frac{1}{2^M}\\
    &\leq \frac{1}{2}\left(\delta^{-\gamma}+1\right)^p\left(\delta^\theta\right)^{\frac{1}{2}}|U|^{\frac{1}{2}}+\frac{1}{2^{\delta^{-\gamma}}}\\
    &\leq \frac{1}{2}\left(\left(2\delta^{-\gamma}\right)^p+\left(2\right)^p\right)\left(\delta^\theta\right)^{\frac{1}{2}}|U|^{\frac{1}{2}}+\frac{1}{2^{\delta^{-\gamma}}}\\
    &= 2^{p-1} \left(\delta^{\frac{\theta}{4}}+\delta^{\frac{\theta}{2}}\right)|U|^{\frac{1}{2}} + \frac{1}{2^{\delta^{-\gamma}}}\\
    &\leq 2^{p-1} \left(2\right)|U|^{\frac{1}{2}} + \delta^\frac{1}{2}\\
    &\leq \left(2^{p} + 1\right) |U|^{\frac{1}{2}},\\
\end{align*}
so by the mass distribution principle \cite[Proposition 2.2]{intermediate}, $\lodim_\theta(E)\geq\frac{1}{2}$ for all $\theta\in(0,1]$.
\end{proof}

\section{Attenuated topologist's sine curve}

The set $\left\{\left(x,\sin\left(\frac{1}{x}\right)\right):x\in(0,1]\right\}\cup\{(0,0)\}$ (Figure \ref{atten}) is known as the topologist's sine curve as it displays interesting properties under the standard Euclidean topology \cite[pp. 137--138]{counterexamples}.

\begin{figure}
    \centering
    \begin{subfigure}[h]{6cm}
        \centering
        \includegraphics[width=3cm]{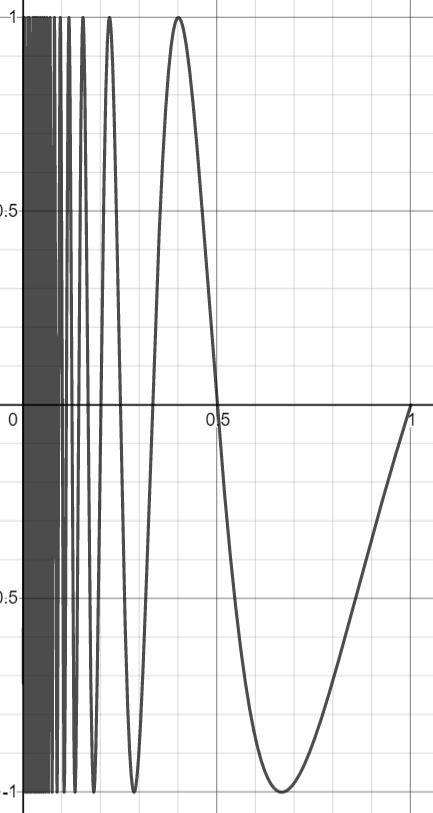} 
        \caption{The set $T_1\subseteq\mathbb{R}^2$.} \label{atten}
    \end{subfigure}
    \begin{subfigure}[h]{6cm}
        \centering
        \includegraphics[width=3cm]{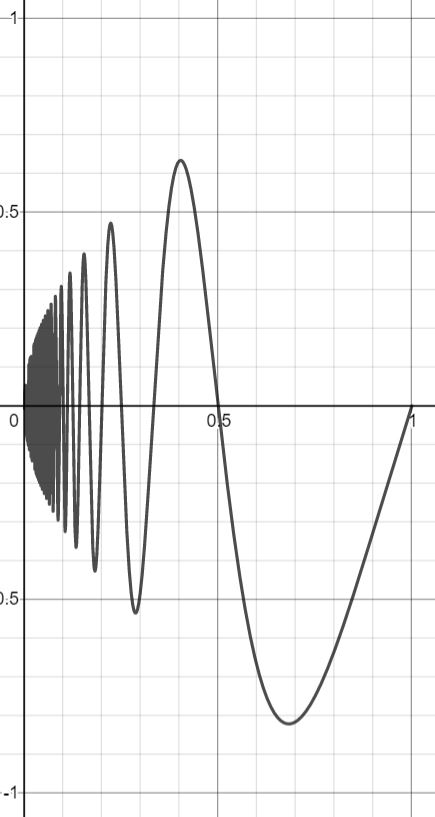}
        \caption{The set $T_{1,0.5}\subseteq\mathbb{R}^2$.} \label{atten2}
    \end{subfigure}
    \caption{Illustrations of (\subref{atten}) the topologist's sine curve and (\subref{atten2}) the attenuated topologist's sine curve.}
\end{figure}

For $p>0$, let
\begin{equation}
    T_p:=\left\{\left(\frac{1}{t^p},\sin(\pi t)\right):t\geq1\right\}.
\end{equation}
That is, $T_p$ is the graph $y=f(x)$ of the function $f:(0,1]\to\mathbb{R}$ defined by $f(x)=\sin\left(\pi x^{-1/p}\right)$.

This looks a lot like the Cartesian product $F_p\times[-1,1]$, and indeed, the intermediate dimensions of this set follow the product formula \cite[Proposition 2.5]{intermediate}. That is, for $p>0$ and $0\leq\theta\leq1$,
\begin{equation}
    \lodim_\theta (T_p) = \updim_\theta (T_p) = 1+\frac{\theta}{\theta+p}=\frac{2\theta+p}{\theta+p}.
\end{equation}

We construct a more interesting set, this time with two parameters, by attenuating the amplitude of the topologist's sine curve.

For $p>0$ and $q>0$, let
\begin{equation}
    T_{p,q}:=\left\{\left(\frac{1}{t^p},\frac{1}{t^{pq}}\sin(\pi t)\right):t\geq1\right\}.
\end{equation}
That is, $T_{p,q}$ is the graph $y=f(x)$ of the function $f:(0,1]\to\mathbb{R}$ defined by $f(x)=x^q\sin\left(\pi x^{-1/p}\right)$. Figure \ref{atten2} illustrates the attenuated topologist's sine curve with $p=1$ and $q=0.5$.

%\begin{figure}[h]
%\caption{The set $T_{1,0.5}$.}
%\label{atten2}
%\includegraphics[width=3cm]{atten2}
%\centering
%\end{figure}

\begin{thm}\label{attenuated}
For $p>0$, $q>0$, and $0\leq\theta\leq1$,
\begin{equation}\label{attenuatedformula}
    \lodim_\theta (T_{p,q}) = \updim_\theta (T_{p,q}) =
    \begin{cases}
        \frac{p(1+q)+2\theta(1-pq)}{p(1+q)+\theta(1-pq)} &\text{if } 0<pq<1\\
        1 &\text{if } pq\geq1
    \end{cases}.
\end{equation}
\end{thm}

Before going into the proof, we wish to draw the reader's attention to an unexpected connection between the sets of concentric spheres in $\mathbb{R}^d$ and the attenuated topologist's sine curve. Let $d\geq2$. For any $p>0$ and $0\leq\theta\leq1$,
\begin{equation}
    \dim_\theta(C_p^d) -(d-1)=\dim_\theta(T_{p,(d-1)}) -1.
\end{equation}
This suggests that attenuating the topologist's sine curve by the polynomial factor $q=d-1$ mimics viewing the set of concentric spheres from $d$-dimensional ambient Euclidean space.

The formula (\ref{attenuatedformula}) also looks strikingly similar to the the formula for the intermediate dimensions of elliptical polynomial spirals, calculated by Burrell, Falconer, and Fraser \cite{elliptical}. For $q\geq p>0$, let
\begin{equation}
    S_{p,q}:=\left\{\left(\frac{1}{t^p}\sin(\pi t),\frac{1}{t^q}\cos(\pi t)\right):t\geq1\right\}\subseteq\mathbb{R}^2.
\end{equation}
Figure \ref{spiral2} illustrates the case with $p=1$ and $q=2$.

%\begin{figure}[h]
%\centering
%\includegraphics[width=5cm]{spiral2}
%\caption{The set $S_{1,2}$.}
%\label{spiral2}

%\end{figure}

Then 
\begin{equation}
    \lodim_\theta \left(S_{p,q}\right) = \updim_\theta \left(S_{p,q}\right) =
    \begin{cases}
        \frac{p+q+2\theta(1-p)}{p+q+\theta(1-p)} &\text{if } 0<p<1\\
        1 &\text{if } p\geq1
    \end{cases}
\end{equation}
\cite[Theorem 2.1]{elliptical}.
Therefore, for $q\geq p>0$, we have $\updim_\theta \left(S_{p,q}\right) = \updim_\theta (T_{q,p/q})$ for all $0\leq\theta\leq1$, and similarly for the lower intermediate dimensions.

\begin{proof}

\emph{Case 1: $pq\geq1$.}
It is clear that $\dim_H(T_{p,q})=1$, since $T_{p,q}$ is a countable union of rectifiable curves. Therefore for all $0\leq\theta\leq1$, $\lodim_\theta(T_{p,q})\geq1$. We now construct a $\delta$-cover to show that $\updim_B(T_{p,q})\leq1$. Let $0<\delta<1$ and $M=\left\lceil\delta^{-\frac{1}{p(1+q)}}\right\rceil$. It is obvious that $\left\{\left(\frac{1}{t^p},\frac{1}{t^{pq}}\sin(\pi t)\right):t\geq M\right\} \subseteq \left\{(x,y(x)):0\leq x\leq\frac{1}{M^p},-x^q\leq y(x)\leq x^q\right\}$, which can be covered by at most $\sum_{i=1}^N\left(\frac{2(i\delta/\sqrt{2})^q}{\delta/\sqrt{2}}+1\right)$ squares of side length $\delta/\sqrt{2}$ in a grid of increasing height, where $N=\left\lceil\frac{\sqrt{2}}{\delta M^p}\right\rceil$. Now, we have:
\begin{align*}
    \sum_{i=1}^N\left(\frac{2(i\delta/\sqrt{2})^q}{\delta/\sqrt{2}}+1\right) &= N+ \frac{2\sqrt{2}^{1-q}\delta^q}{\delta} \sum_{i=1}^Ni^q\\
    &\leq N+ \frac{2\sqrt{2}^{1-q}\delta^q}{\delta} \int_{x=1}^{N+1} x^q dx\\
    &\leq N+ \frac{2\sqrt{2}^{1-q}}{1+q} \frac{\delta^q (N+1)^{1+q}}{\delta}\\
    &\leq \left(\frac{\sqrt{2}}{\delta M^p}+1\right) + \frac{2\sqrt{2}^{1-q}}{1+q} \frac{\delta^q (\frac{\sqrt{2}}{\delta M^p}+2)^{1+q}}{\delta}\\
    &\leq \sqrt{2}\delta^{\frac{1}{1+q}-1}+ \frac{2\sqrt{2}^{1-q}}{1+q} \frac{\delta^q (\sqrt{2}\delta^{\frac{1}{1+q}-1}+2)^{1+q}}{\delta} +1\\
    &\leq \sqrt{2}\delta^{\frac{1}{1+q}-1}+ \frac{2\sqrt{2}^{1-q}}{1+q} \frac{\delta^q 2^{\frac{3}{2}(1+q)} \delta^{(\frac{1}{1+q}-1)(1+q)}}{\delta} + \frac{2\sqrt{2}^{1-q}}{1+q} \frac{\delta^q 2^{1+q} 2^{1+q}}{\delta} +1\\
    &= \sqrt{2}\delta^{-\frac{q}{1+q}}+ \frac{2^{3+q}}{1+q} \delta^{-1} + \frac{2^{\frac{7+3q}{2}}}{1+q} \delta^{q-1} +1.\\
\end{align*}

We also calculate an upper bound on the length of $\left\{\left(\frac{1}{t^p},\frac{1}{t^{pq}}\sin(\pi t)\right):1\leq t\leq M\right\}$:
\begin{align*}
    \mathcal{H}^1 \left(\left\{\left(\frac{1}{t^p},\frac{1}{t^{pq}}\sin(\pi t)\right):1\leq t\leq M\right\}\right) &= \sum_{i=1}^{M-1} \mathcal{H}^1 \left(\left\{\left(\frac{1}{t^p},\frac{1}{t^{pq}}\sin(\pi t)\right):i\leq t\leq i+1\right\}\right)\\
    &\leq \sum_{i=1}^{M-1} \left(\frac{2}{i^{pq}}+\frac{1}{i^p}-\frac{1}{(i+1)^p}\right)\\
    &= 2 +\sum_{i=2}^{M-1} \frac{2}{i^{pq}} +\frac{1}{1^p}-\frac{1}{M^p}\\
    &\leq 2+ \sum_{i=2}^{M-1} \frac{2}{i} +1\\
    &\leq \int_{x=1}^{M-1} \frac{2}{x} dx +3\\
    &= 2\log(M-1) +3\\
    &\leq 2\log(\delta^{-\frac{1}{p(1+q)}}) +3\\
    &= \frac{2}{p(1+q)}\log\left(\frac{1}{\delta}\right) +3.\\
\end{align*}
Therefore, we can cover $\left\{\left(\frac{1}{t^p},\frac{1}{t^{pq}}\sin(\pi t)\right):1\leq t\leq M\right\}$ by at most $\frac{1}{\delta}\left(\frac{2}{p(1+q)}\log\left(\frac{1}{\delta}\right) +3\right) +1$ sets of diameter $\delta$.

This gives us a $\delta$-cover of $T_{p,q}$, proving that for $0<\delta<1$,
\begin{align*}
    N_\delta\left(T_{p,q}\right) &\leq \frac{1}{\delta}\left(\frac{2}{p(1+q)}\log\left(\frac{1}{\delta}\right) +3\right) +1 + \sqrt{2}\delta^{-\frac{q}{1+q}}+ \frac{2^{3+q}}{1+q} \delta^{-1} + \frac{2^{\frac{7+3q}{2}}}{1+q} \delta^{q-1} +1\\
    &= \frac{1}{\delta}\left(\frac{2}{p(1+q)}\log\left(\frac{1}{\delta}\right) +\frac{2^{3+q}}{1+q} +3\right) + \sqrt{2}\delta^{-\frac{q}{1+q}} + \frac{2^{\frac{7+3q}{2}}}{1+q} \delta^{q-1} +2.\\
\end{align*}

Now applying the definition of the upper box dimension, we have
\begin{align*}
    \updim_B\left(T_{p,q}\right) &= \limsup_{\delta\to0^+}\frac{\log N_\delta(T_{p,q})}{-\log(\delta)}\\
    &\leq \limsup_{\delta\to0^+} \frac{\log \left( \frac{1}{\delta}\left(\frac{2}{p(1+q)}\log\left(\frac{1}{\delta}\right) +\frac{2^{3+q}}{1+q} +3\right) + \sqrt{2}\delta^{-\frac{q}{1+q}} + \frac{2^{\frac{7+3q}{2}}}{1+q} \delta^{q-1} +2 \right)}{\log\left(\frac{1}{\delta}\right)}\\
    &\leq \limsup_{\delta\to0^+} \frac{\log \left(\frac{1}{\delta}\left(\frac{2}{p(1+q)}\log\left(\frac{1}{\delta}\right) +\frac{2^{3+q}}{1+q} +3\right) \right)}{\log\left(\frac{1}{\delta}\right)}\\
    &= \limsup_{\delta\to0^+} \frac{\log \left(\frac{1}{\delta}\right)+\log\left(\frac{2}{p(1+q)}\log\left(\frac{1}{\delta}\right) +\frac{2^{3+q}}{1+q} +3 \right)}{\log\left(\frac{1}{\delta}\right)} =1.
\end{align*}

\emph{Case 2: $0<pq<1$.}
First, we construct appropriate measures so as to derive a lower bound using the mass distribution principle. Write $s=\frac{p(1+q)+2\theta(1-pq)}{p(1+q)+\theta(1-pq)}$, let $0<\delta<4^{-\frac{1+p}{(s-\theta s + 2 \theta -1)(1-pq)}}$ be given, and let $M=\left\lceil\delta^{-\frac{s-\theta s +2\theta -1}{1+p}}\right\rceil$. Define $\mu_\delta:= \delta^{s-1} \sum_{i=1}^{M-1} \mathcal{H}^1\restriction_{\left\{\left(\frac{1}{t^p},\frac{1}{t^{pq}}\sin(\pi t)\right):i\leq t\leq i+1\right\}}$. That is, for some $U\subseteq\mathbb{R}^2$, $\mu_\delta(U)$ is the length of the intersection between $U$ and the first $M$ arcs (counting from the right) of $T_{p,q}$, scaled by the factor $\delta^{s-1}$.

The total mass carried by $T_{p,q}$ is then:
\begin{align*}
    \mu_\delta(T_{p,q}) &= \delta^{s-1} \sum_{i=1}^{M-1} \mathcal{H}^1\restriction_{\left\{\left(\frac{1}{t^p},\frac{1}{t^{pq}}\sin(\pi t)\right):i\leq t\leq i+1\right\}}(T_{p,q})\\
    &= \delta^{s-1} \sum_{i=1}^{M-1} \mathcal{H}^1\left({\left\{\left(\frac{1}{t^p},\frac{1}{t^{pq}}\sin(\pi t)\right):i\leq t\leq i+1\right\}}\right)\\
    &\geq \delta^{s-1} \sum_{i=2}^{M} \frac{1}{i^{pq}}\\
    &\geq \delta^{s-1} \int_{x=2}^{M} \frac{1}{x^{pq}} dx\\
    &= \frac{\delta^{s-1}}{1-pq} \left(M^{1-pq}-2^{1-pq}\right)\\
    &\geq \frac{\delta^{s-1}}{1-pq} \frac{M^{1-pq}}{2}\\
    &(\text{since }M^{1-pq}>4\text{ and }2^{1-pq}<2)\\
    &\geq \frac{1}{2(1-pq)} \delta^{s-1} \delta^{-\frac{s-\theta s + 2 \theta -1}{1+p}(1-pq)}\\
    &= \frac{1}{2(1-pq)} \delta^{\frac{(s-1)(1+p)-s(1-pq)+\theta s(1-pq) - 2\theta(1-pq) +(1-pq)}{1+p}}\\
    &= \frac{1}{2(1-pq)} \delta^{\frac{s[p(1+q)+\theta(1-pq)]-[p(1+q) + 2\theta(1-pq)]}{1+p}}\\
    &= \frac{1}{2(1-pq)},
\end{align*}
which is independent of our choice of $0<\delta<4^{-\frac{1+p}{(s-\theta s + 2 \theta -1)(1-pq)}}$.

Suppose that $U\subseteq\mathbb{R}^2$ is such that $\delta\leq|U|\leq\delta^\theta$. Then $U$ intersects at most $\frac{|U|M^{1+p}}{p}+2$ of the arcs $\left\{\left(\frac{1}{t^p},\frac{1}{t^{pq}}\sin(\pi t)\right):i\leq t\leq i+1\right\}$ (for $1\leq i\leq M-1$) (by the mean value theorem). Based on the shape of the arcs (sine curve), for each arc that $U$ intersects, the length of this intersection does not exceed $3|U|$. Thus, the mass that $|U|$ carries can be bounded above by:
\begin{align*}
    \mu_\delta(U) &\leq \delta^{s-1}\left(\frac{|U|M^{1+p}}{p}+2\right)(3|U|)\\
    &\leq \delta^{s-1}\left(\frac{|U|\left(\delta^{-\frac{s-\theta s +2\theta -1}{1+p}}+1\right)^{1+p}}{p}+2\right)(3|U|)\\
    &\leq \delta^{s-1}\left(\frac{|U|\left(2\delta^{-\frac{s-\theta s +2\theta -1}{1+p}}\right)^{1+p}}{p}+2\right)(3|U|)\\
    &= 3 \delta^{s-1} \left(\frac{2^{1+p}\delta^{-s+\theta s -2\theta +1}}{p}|U|^{2-s}+2|U|^{1-s}\right)|U|^s\\
    &\leq 3 \delta^{s-1} \left(\frac{2^{1+p}\delta^{-s+\theta s -2\theta +1}}{p}\delta^{\theta(2-s)}+2\delta^{\theta(1-s)}\right)|U|^s\\
    &= 3 \left(\frac{2^{1+p}}{p}\delta^{1-s+s-1} + 2\delta^{(1-\theta)(s-1)}\right)|U|^s\\
    &\leq 3\left(\frac{2^{1+p}}{p}+2\right)|U|^s\\
    &(\text{since }s-1=\frac{\theta(1-pq)}{p(1+q)+\theta(1-pq)}\geq0).\\
\end{align*}

By the mass distribution principle \cite[Proposition 2.2]{intermediate}, we conclude that $\lodim_\theta(T_{p,q})\geq s=\frac{p(1+q)+2\theta(1-pq)}{p(1+q)+\theta(1-pq)}$.

Moving on, we construct an efficient cover. Taking $0<\delta<1$ as given, suppose that $s>\frac{p(1+q)+2\theta(1-pq)}{p(1+q)+\theta(1-pq)}$, and let $M=\left\lceil\delta^{-\frac{s-\theta s +2\theta -1}{1+p}}\right\rceil$. We observe that \\$\left\{\left(\frac{1}{t^p},\frac{1}{t^{pq}}\sin(\pi t)\right):t\geq M\right\}\subseteq \left\{(x,y(x)):0\leq x\leq\frac{1}{M^p},-x^q\leq y(x)\leq x^q\right\}$ can be covered by at most $\sum_{i=1}^N\left(\frac{2(i\delta^\theta/\sqrt{2})^q}{\delta^\theta/\sqrt{2}}+1\right)$ squares of side length $\delta^\theta/\sqrt{2}$ (and hence diameter $\delta^\theta$), where $N=\left\lceil\frac{\sqrt{2}}{\delta^\theta M^p}\right\rceil$. We have:
\begin{align*}
    \sum_{i=1}^N \left(\frac{2(i\delta^\theta/\sqrt{2})^q}{\delta^\theta/\sqrt{2}}+1\right) &= N + 2\sqrt{2}^{1-q}\delta^{\theta(q-1)} \sum_{i=1}^N i^q\\
    &\leq N + 2\sqrt{2}^{1-q}\delta^{\theta(q-1)} \int_{x=1}^{N+1} x^q dx\\
    &\leq N + \frac{2\sqrt{2}^{1-q}}{1+q}\delta^{\theta(q-1)} (N+1)^{1+q}\\
    &\leq \left(\frac{\sqrt{2}}{\delta^\theta M^p}+1\right) + \frac{2\sqrt{2}^{1-q}}{1+q}\delta^{\theta(q-1)} \left(\frac{\sqrt{2}}{\delta^\theta M^p}+2\right)^{1+q}\\
    &\leq \frac{\sqrt{2}}{\delta^\theta \delta^{-\frac{s-\theta s +2\theta -1}{1+p}p}} + \frac{2\sqrt{2}^{1-q}}{1+q}\delta^{\theta(q-1)} \left(\frac{\sqrt{2}}{\delta^\theta \delta^{-\frac{s-\theta s +2\theta -1}{1+p}p}}+2\right)^{1+q}+1\\
    &= \sqrt{2} \delta^{\frac{sp-\theta sp+\theta p-p-\theta}{1+p}} + \frac{2\sqrt{2}^{1-q}}{1+q}\delta^{\theta(q-1)} \left(\sqrt{2}\delta^{\frac{sp-\theta sp+\theta p-p-\theta}{1+p}}+2\right)^{1+q}+1\\
    &\leq \sqrt{2} \delta^{\frac{sp-\theta sp+\theta p-p-\theta}{1+p}} + \frac{2\sqrt{2}^{1-q}}{1+q}\delta^{\theta(q-1)} 2^{\frac{3}{2}(1+q)}\delta^{\frac{sp-\theta sp+\theta p-p-\theta}{1+p}(1+q)}\\
    &+ \frac{2\sqrt{2}^{1-q}}{1+q}\delta^{\theta(q-1)}2^{1+q}2^{1+q}+1\\
    &= \sqrt{2} \delta^{\frac{sp-\theta sp+\theta p-p-\theta}{1+p}} + \frac{2^{3+q}}{1+q}\delta^{\frac{- 2\theta + 2\theta pq + sp-\theta sp-p + spq -\theta spq -pq}{1+p}} + \frac{2^{\frac{7+3q}{2}}}{1+q}\delta^{\theta(q-1)}+1.\\
\end{align*}

Next, we calculate an upper bound on the length of $\left\{\left(\frac{1}{t^p},\frac{1}{t^{pq}}\sin(\pi t)\right):1\leq t\leq M\right\}$:
\begin{align*}
    \mathcal{H}^1 \left(\left\{\left(\frac{1}{t^p},\frac{1}{t^{pq}}\sin(\pi t)\right):1\leq t\leq M\right\}\right) &= \sum_{i=1}^{M-1} \mathcal{H}^1 \left(\left\{\left(\frac{1}{t^p},\frac{1}{t^{pq}}\sin(\pi t)\right):i\leq t\leq i+1\right\}\right)\\
    &\leq \sum_{i=1}^{M-1} \left(\frac{2}{i^{pq}}+\frac{1}{i^p}-\frac{1}{(i+1)^p}\right)\\
    &= \sum_{i=1}^{M-1} \frac{2}{i^{pq}}+\frac{1}{1^p}-\frac{1}{M^p}\\
    &\leq \int_{x=0}^{M-1} \frac{2}{x^{pq}} dx + 1\\
    &= \frac{2}{1-pq}(M-1)^{1-pq} + 1\\
    &\leq \frac{2}{1-pq}\delta^{-\frac{s-\theta s +2\theta -1}{1+p}(1-pq)} + 1\\
    &= \frac{2}{1-pq}\delta^{\frac{spq-\theta spq +2\theta pq - pq -s+\theta s -2\theta +1}{1+p}} + 1.\\
\end{align*}
Thus, we can cover $\left\{\left(\frac{1}{t^p},\frac{1}{t^{pq}}\sin(\pi t)\right):1\leq t\leq M\right\}$ with at most $\frac{1}{\delta} \left(\frac{2}{1-pq}\delta^{\frac{spq-\theta spq +2\theta pq - pq -s+\theta s -2\theta +1}{1+p}} + 1\right)+1 = \frac{2}{1-pq}\delta^{\frac{spq-\theta spq +2\theta pq - pq -s+\theta s -2\theta -p}{1+p}}+\delta^{-1}+1$ sets of diameter $\delta$.

This gives us a cover $\mathcal{U}$ of $T_{p,q}$ such that $\delta\leq|U|\leq\delta^\theta$ for all $U\in\mathcal{U}$. Summing over all the sets in this cover, we have:
\begin{align*}
    \sum_{U\in\mathcal{U}}|U|^s &\leq \left(\sqrt{2} \delta^{\frac{sp-\theta sp+\theta p-p-\theta}{1+p}} + \frac{2^{3+q}}{1+q}\delta^{\frac{- 2\theta + 2\theta pq + sp-\theta sp-p + spq -\theta spq -pq}{1+p}} + \frac{2^{\frac{7+3q}{2}}}{1+q}\delta^{\theta(q-1)}+1 \right) \delta^{\theta s}\\
    &+ \left( \frac{2}{1-pq}\delta^{\frac{spq-\theta spq +2\theta pq - pq -s+\theta s -2\theta -p}{1+p}}+\delta^{-1} +1\right) \delta^s\\
    &= \sqrt{2} \delta^{\frac{sp-\theta sp+\theta p-p-\theta+\theta s + \theta sp}{1+p}} + \frac{2^{3+q}}{1+q}\delta^{\frac{- 2\theta + 2\theta pq + sp-\theta sp-p + spq -\theta spq -pq + \theta s + \theta sp}{1+p}}+ \frac{2^{\frac{7+3q}{2}}}{1+q}\delta^{\theta(s-1+q)}+\delta^{\theta s}\\
    &+ \frac{2}{1-pq}\delta^{\frac{spq-\theta spq +2\theta pq - pq -s+\theta s -2\theta -p +s+sp}{1+p}}+\delta^{s-1} + \delta^s\\
    &= \sqrt{2} \delta^{\frac{sp+\theta p-p-\theta+\theta s}{1+p}} + \frac{2^{3+q}}{1+q}\delta^{\frac{- 2\theta + 2\theta pq + sp-p + spq -\theta spq -pq +\theta s}{1+p}}+ \frac{2^{\frac{7+3q}{2}}}{1+q}\delta^{\theta(s-1+q)}+\delta^{\theta s}\\
    &+ \frac{2}{1-pq}\delta^{\frac{spq-\theta spq +2\theta pq - pq +\theta s -2\theta -p +sp}{1+p}}+\delta^{s-1} + \delta^s\\
    &= \sqrt{2} \delta^{\frac{(p+\theta)(s-1)+\theta p}{1+p}} + \frac{2^{3+q}}{1+q}\delta^{\frac{s[p(1+q)+\theta(1-pq)] -[p(1+q)+2\theta(1-pq)]}{1+p}} + \frac{2^{\frac{7+3q}{2}}}{1+q}\delta^{\theta(s-1+q)}+\delta^{\theta s}\\
    &+ \frac{2}{1-pq}\delta^{\frac{s[p(1+q)+\theta(1-pq)]-[p(1+q)+2\theta(1-pq)]}{1+p}}+\delta^{s-1} + \delta^s,\\
\end{align*}
which converges to $0$ as $\delta\to0^+$ under our assumption that $s>\frac{p(1+q)+2\theta(1-pq)}{p(1+q)+\theta(1-pq)}=1+\frac{\theta(1-pq)}{p(1+q)+\theta(1-pq)}\geq1$ (which ensures that all the powers of $\delta$ in the expression are positive). This proves that $\updim_\theta(T_{p,q})\leq s$. But this holds for any $s>\frac{p(1+q)+2\theta(1-pq)}{p(1+q)+\theta(1-pq)}$, so we conclude that $\updim_\theta(T_{p,q})\leq \frac{p(1+q)+2\theta(1-pq)}{p(1+q)+\theta(1-pq)}$.
\end{proof}

In the case considered in Theorem \ref{attenuated}, the envelope with which we attenuated the topologist's sine curve was defined by the function $g(x)=x^q$. One might then ask how the intermediate dimensions will behave if we use different enveloping functions. Based on the proof of Theorem \ref{attenuated} and taking $q\geq1/p$ and the limit $q\to0^+$, it is easy to see that if the enveloping function is eventually narrower than any polynomial function (as $x\to0^+$), then intermediate dimensions will be minimal (that is, 1); and if the enveloping function is eventually wider than any polynomial function (as $x\to0^+$), then the attenuation will not reduce the intermediate dimensions (relative to the topologist's sine curve).

This tells us that, for example, for all $p>0$, we have
\begin{equation*}
    \dim_\theta\left(\left\{\left(x,2^{-1/x}\sin\left(\pi x^{-1/p}\right)\right):x\in\left(0,1\right]\right\}\right) = 1
\end{equation*}
and
\begin{equation*}
    \dim_\theta\left(\left\{\left(x,\frac{1}{-\log(x)}\sin\left(\pi x^{-1/p}\right)\right):x\in\left(0,\frac{1}{3}\right]\right\}\right) = 1+\frac{\theta}{\theta+p}
\end{equation*}
for all $\theta\in[0,1]$.

\section*{Acknowledgements}
This note is based on the author's MSc Mathematics dissertation project, and the author is very grateful to his supervisor, Dr Jonathan Fraser, for his suggestions and feedback. The author would also like to thank Amlan Banaji for many helpful discussions.

\vfill

\footnotesize
{\parindent0pt
\textsc{Justin T. Tan\\ School of Mathematics and Statistics \\ The University of St Andrews \\ St Andrews, KY16 9SS, Scotland}\par\nopagebreak
\textit{Email:} \texttt{jtt1@st-andrews.ac.uk}
}

\end{document}